\documentclass[12pt]{amsart}
\usepackage{txfonts}
\usepackage{amscd,amssymb}
\usepackage[colorlinks,plainpages,backref,urlcolor=blue]{hyperref}
\usepackage{verbatim}
\usepackage{xcolor}

\newtheorem{theorem}{Theorem}[section]
\newtheorem{corollary}[theorem]{Corollary}

\newtheorem{prop}[theorem]{Proposition}

\theoremstyle{definition}
\newtheorem{definition}[theorem]{Definition}
\newtheorem{example}[theorem]{Example}
\newtheorem{remark}[theorem]{Remark}

\newtheorem{proposition}[theorem]{Proposition}

\newcommand{\C}{\mathbb{C}}

\DeclareMathAlphabet{\pazocal}{OMS}{zplm}{m}{n}
\newcommand{\A}{{\pazocal{A}}}

\newcommand{\cE}{{\mathcal{E}}}

\newcommand{\CC}{{\mathcal{C}}}

\def\dot{\mathchar"013A}
\newcommand{\hdot}{{\raise1pt\hbox to0.35em{\Huge $\dot$}}}

\begin{document}
\date{\today}

\title[On plus-one generated conic-line arrangements with simple singularities
]%
{On plus-one generated conic-line arrangements with simple singularities
}
\author{Anca~M\u acinic and Piotr Pokora}
\maketitle
\thispagestyle{empty}
\begin{abstract}
In this paper we study plus-one generated arrangements of conics and lines in the complex projective plane with simple singularities. We provide several degree-wise classification results that allow us to construct explicit examples of such arrangements. 
\end{abstract}
\section{Introduction}
Recently introduced by Abe in \cite{A0}, the class of plus-one generated  arrangements of hyperplanes proved to be strongly connected to  the class of free hyperplane arrangements, a connection that we will explain shortly.

  Let us recall that an arrangement of hyperplanes is called {\it free} if its associated module of derivations is a free module, over the coordinate ring (see \cite{Y} for a comprehensive survey on the freeness for hyperplane arrangements).  For a plus-one generated arrangement, the associated module of derivations is no longer free, but it admits a very simple minimal free resolution, so it is, in a way, a natural step away from the freeness property. These definitions are easily rephrased for curves, via associated modules of derivations, see Definitions \ref{def:free_curve}, \ref{def:POG_curves} in Section \ref{sec:preliminaries} for details.

In the case of projective line arrangements, the plus-one generated property appears in close relation to freeness. More precisely, if one deletes a line from a free arrangement, then the resulting arrangement is either free or plus-one generated, and the same goes for the addition of a line, see Theorem \ref{thm:Abe}. 
When passing to higher dimensions, i.e., for arrangements of hyperplanes, the result of the deletion from a free arrangement is still either free or plus-one generated. On the other hand, the addition of a hyperplane to a free arrangement does not produce only free or plus-one generated arrangements, see \cite{A0}. However, in \cite{ADen}, a 'dual' notion of plus-one generated is introduced, based on the algebra of logarithmic differential forms on a hyperplane arrangement, and the addition behaves well with respect to this dual notion.

One can wonder whether the same kind of behavior occurs for reduced plane curves, or at least for certain types of reduced curves, such as conic-line arrangements. 
Recently proven results from \cite{dim2, MP} show that the deletion or addition of a line from/to a free conic-line arrangement produces a free or a plus-one generated conic-line arrangement.
Our examples so far seem to support the hypothesis that deleting a conic from a conic-line arrangement results in either a free or a plus-one generated conic-line arrangement. Addition of a conic does not follow the same pattern, as illustrated by Example \ref{addit}.

In Section \ref{sec:preliminaries} we make an inventory of notions and results in the field that we rely on, and present a set of examples relevant to the relation between being free and plus-one generated for conic-line arrangements. 

 In Section \ref{sec:main} we present classification results for plus-one generated conic-line arrangements, under some restrictions on the types of singularities and value of the defect (see Section \ref{sec:preliminaries} for the definition). One finds that such restrictions limit severely the number of irreducible components of such arrangements and we can compare these results with similar ones on free and nearly-free arrangements from \cite{DJP, P}.
In Theorem \ref{thm:C_arr}, we consider the case of plus-one generated arrangements of conics with some simple singularities and the defect equal to $2$, which is the smallest possible value for a defect that does not produce a nearly-free curve, and we prove that such an arrangement can contain at most $4$ conics. 
In Theorem \ref{thm:CL_arr}, we work with conic-line arrangements having at least one line and one conic with simple singularities and defect $2$, to reach the conclusion that the number of irreducible components of such an arrangement can be at most $9$. The added value of our work is the number of examples of plus-one generated conic-line arrangements that appear in our proofs.

\section{Preliminaries}
\label{sec:preliminaries}
\label{sec:del}
Let us start with a general introduction. Denote by  $f \in S =  \C[x,y,z]$ the defining polynomial of a reduced plane curve $\CC \, : f=0$ in $\mathbb{P}^{2}_{\mathbb{C}}$ such that $f = f_{1} \cdots f_{k}$ and ${\rm GCD}(f_{i}, f_{j})=1$ for $i\neq j$. It means that $\CC$ consists of $k$ irreducible components $\CC= \{C_{1}, ..., C_{k}\}$ with $C_{i} \, : f_{i}=0$.  Let ${\rm Der}(S) = S \cdot \partial_{x} \oplus S\cdot \partial_{y} \oplus S\cdot \partial_{z}$. We define $D(\CC)$ to be the derivation module associated with $\CC$, namely
$$D(\CC) = \{ \theta \in {\rm Der}(S) \, : \, \theta (f) \in \langle f \rangle \}.$$
Since for every element $\theta \in {\rm Der}(S)$ we have
$$\theta(f_{1} \cdots f_{k}) = f_{1} \theta(f_{2} \cdots f_{k}) + f_{2} \cdots f_{k}\theta(f_{1}),$$
then by the inductive application of the above identity, we obtain also
\begin{equation}
\label{deriv}
 D(\CC) = \{ \theta \in {\rm Der}(S) \, : \, \theta(f_{i}) \in \langle f_{i} \rangle, \,\, i =1, \ldots, k\}.   
\end{equation}
It is well-known that
$$D(\mathcal{C}) = S\cdot \delta_{E} \oplus D_{0}(\mathcal{C}),$$
where $\delta_{E}$ denotes the Euler derivation, i.e., $\delta_{E} := x\partial_{x} + y \partial_{y} + z \partial_{z}$ and 
$ D_{0}(\mathcal{C})$ is the submodule of $ D(\mathcal{C})$ defined by

 $$
 D_{0}(\CC) := \{\theta \in Der(S) \; : \; \theta(f) = 0\}.
 $$

\begin{definition}
\label{def:free_curve}
We say that a reduced plane curve $\CC \, : f=0$ with $f \in S_{d}$ for $d\geq 1$ is \textbf{free} if $D_{0}(\CC) = S(-d_{1})\oplus S(-d_{2})$ with $d_{1}\leq d_{2}$. The pair ${\rm exp}(\CC) = (d_{1},d_{2})$ is called the exponents of $\CC$.
\end{definition}
In order to introduce the second most important class of reduced curves in our investigations, we need the following general definition.
\begin{definition}
We say that a reduced plane curve $\CC$ is an $m$-syzygy curve when the associated Milnor algebra $M(f)$ has the following minimal graded free resolution:
$$0 \rightarrow \bigoplus_{i=1}^{m-2}S(-e_{i}) \rightarrow \bigoplus_{i=1}^{m}S(1-d - d_{i}) \rightarrow S^{3}(1-d)\rightarrow S \rightarrow M(f) \rightarrow 0$$
with $e_{1} \leq e_{2} \leq ... \leq e_{m-2}$ and $1\leq d_{1} \leq ... \leq d_{m}$. The $m$-tuple $(d_{1}, ..., d_{m})$ is called the exponents of $\CC$.
\end{definition}

\begin{definition}
\label{def:POG_curves}
A reduced curve $\CC$ in $\mathbb{P}^{2}_{\mathbb{C}}$ is called \textbf{plus-one generated} with the exponents $(d_1,d_2), d_1 \leq d_2$, and level $d_{3}$ if $D_0(\CC)$ admits a minimal resolution of the form:
$$0 \rightarrow  S(-d_3-1) \rightarrow  S(-d_3) \oplus S(-d_2) \oplus S(-d_1) \rightarrow D_0(\CC) \rightarrow 0$$
\end{definition}
\begin{remark}
\begin{enumerate}
\item A $3$-syzygy reduced curve $\CC$ in $\mathbb{P}^{2}_{\mathbb{C}}$ of degree $d$ and exponents $(d_1,d_2,d_3)$ such that $d_{1}+d_{2}=d$ is precisely a plus-one generated curve with exponents $(d_{1}, d_{2})$ and level $d_{3}$. 
\item If $\CC$ is a plus-one generated curve with $d_{2}=d_{3}$, then $\CC$ is called \textbf{nearly-free}.
\end{enumerate}
\end{remark}
We will need the following characterization of plus-one generated reduced plane curves that comes from \cite{DS0}. Here by $\tau(\CC)$ we denote the total Tjurina number of a given reduced curve $\CC \subset \mathbb{P}^{2}_{\mathbb{C}}$, namely
$$\tau(\CC) = \sum_{p \in {\rm Sing}(\CC)} \tau_{p}$$
with $\tau_{p}$ being a local Tjurina number of a singular point $p \in \CC$.
\begin{proposition}[Dimca-Sticlaru]
\label{dimspl}
Let $\CC : f=0$ be a reduced $3$-syzygy curve of degree $d\geq 3$ with the exponents $(d_{1},d_{2},d_{3})$. Then $\CC$ is plus-one generated if and only if
$$\tau(\CC) = (d-1)^{2} - d_{1}(d-d_{1}-1) - (d_{3}-d_{2}+1).$$
\end{proposition}
\begin{definition}
The number $\nu(\CC) := (d_{3}-d_{2}+1)$ is called \textbf{the defect} of $\CC$.
\end{definition}
It clear from the definitions of free and nearly-free curves that if $\CC$ is plus-one generated with $d_{3} > d_{2}$, then $\nu(\CC)\geq 2$. From now on, we will focus on the case when $\nu(\CC)=2$.

Assume from now that $\mathcal{C}$ is a CL-arrangement consisting of $k$ smooth conics and $d$ lines. Moreover, we will work with the case where all singularities of $\mathcal{C}$ are quasi-homogeneous, i.e., for every singular point $p \in {\rm Sing}(\mathcal{C})$ one has $\tau_{p} = \mu_{p}$, where $\mu_{p}$ denotes the local Milnor number and $\tau_{p}$ is the local Tjurina number.

In the theory of line arrangements we have some natural techniques that allow to construct new examples of either free or plus-one generated arrangements, namely addition-deletion techniques. More precisely, we have the following result proved by Abe. 
\begin{theorem} \cite[Theorem 1.11]{A0}
\label{thm:Abe}
 Let $\A$ be a free arrangement of lines in  $\mathbb{P}^{2}_{\mathbb{C}}$. Then:
 \begin{enumerate}
     \item For $L \in \A$, the subarrangement $\A \setminus \{L\}$ is either free or plus-one generated.
     \item Let $L$ be a line in $\mathbb{P}^{2}_{\mathbb{C}}$. Then the arrangement $\A \cup \{L\}$ is either free or plus-one generated.
 \end{enumerate}
\end{theorem}

In fact, a deletion type result as above holds in arbitrary dimension for hyperplane arrangements, see \cite[Theorem 1.4]{A0}. However, this is no longer the case for addition, see for instance \cite[Example 7.4]{A0}.\\

In the world of conic-line arrangements in the plane, it is very natural to wonder whether we can use similar addition-deletion techniques to construct new examples of free or plus-one generated arrangements. The main difference is based on the fact that we can add/delete either a conic or a line. By \cite{dim2, MP}, if we have a free conic-line arrangement and we add or delete a line, then the resulting arrangement is either free or plus-one generated. On the other hand, with the  addition of either a conic or a line, we can get the whole spectrum of possibilities. Let us illustrate this slogan with the following two examples.
\begin{example}
\label{XR}
Let us consider the conic-line arrangement $\mathcal{XR}\subset \mathbb{P}^{2}_{\mathbb{C}}$ given by the following defining polynomial:
\begin{multline*}
    Q(x,y,z)= (x^{2} + 2xy + y^{2} - xz)(x^2 + xy + 2yz - z^2 )(x^2 + xz + yz)(x^2 + xy + z^2 ) \\ (x^2 + 2xy - xz + yz)(x^2 - y^2 + xz + 2yz)y(x + z)(x + y - z)(x + y + z)(x - z)\bigg(x + \frac{1}{2}y\bigg).
\end{multline*}
This arrangement has $9$ ordinary sixfold and $12$ double intersection points. It is also worth noticing that all singularities are quasi-homogeneous. Using \verb}SINGULAR} we can check that the arrangement $\mathcal{XR}$ is free with exponents $(d_{1},d_{2})=(4,13)$. Now we perform the addition technique. If we add to $\mathcal{XR}$ the line $\ell : x+2y+4z=0$, then the arrangement is plus-one generated with exponents $(d_{1},d_{2},d_{3})=(5,14,17)$. Looking more precisely, by adding the line $\ell$ we introduce $18$ additional double points to the resulting arrangement. 

However, using the same addition trick, we can obtain a free arrangement of conics and lines, and this can be achieved by adding to $\mathcal{XR}$, for instance, the line $\ell' : z=0$.
\end{example}
\begin{example}
\label{addit}
Let us consider a conic-line arrangement $\mathcal{CL}\subset \mathbb{P}^{2}_{\mathbb{C}}$ given by the following defining polynomial:
$$Q(x,y,z) = (x^2 + y^2 - z^2 )(y-z)(x^2 - z^2 ).$$
This arrangement is known to be free \cite[Example 4.14]{DimcaPokora}. Observe that it has $3$ nodes ($A_{1}$ singularities) and $3$ tacnodes ($A_{3}$ singularities), and this gives us $\tau(\mathcal{CL})=12$. If we add the conic $C \, : y^2-xz=0$ to $\mathcal{CL}$, then, using \verb}SINGULAR}, we can check that the resulting arrangement is only $4$-syzygy.
\end{example}

As we will see in next section, the addition-deletion techniques are not always sufficient in our classification considerations, mainly due to the fact that there are not that many known examples of free CL-arrangements with simple singularities. For example, there are no free conic arrangements with only nodes and tacnodes as singularities \cite[Proposition 1.5]{DJP}, but there are plus-one generated conic arrangements with nodes and tacnodes. This is the main reason why we are forced to use different techniques to obtain our results. On the other hand, our combinatorial techniques can be applied more generally, so we hope that they will be useful in further research.
\section{Plus-one generated arrangements with certain ${\rm ADE}$ singularities}
\label{sec:main}
 Our aim here is to provide a degree-wise characterization of plus-one generated with some prescribed ${\rm ADE}$ singularities. We start with arrangements of conics in the plane. Our result is inspired by a recent paper due to Dimca, Janasz, and the second author devoted to conic arrangements in the plane admitting nodes and tacnodes \cite{DJP}.
\begin{theorem}
\label{thm:C_arr}
Let $\mathcal{C} \subset \mathbb{P}^{2}_{\mathbb{C}}$ be an arrangement of $k\geq 2$ smooth conics such that they admit $n_{2}$ nodes, $t_{2}$ tacnodes, and $n_{3}$ ordinary triple points. Assume that $\mathcal{C}$ is plus-one generated with the defect $\nu(\mathcal{C})=2$, then $k \in \{2,3,4\}$.
\end{theorem}
\begin{proof}
Using Proposition \ref{dimspl}, if $\mathcal{C}$ is plus-one generated of degree $d=2k$ with $k\geq 2$ and $\nu(\mathcal{C})=2$, then we have
$$d_{1}^{2}-d_{1}(2k-1)+(2k-1)^2 = \tau(\mathcal{C})+\nu(\mathcal{C}) = n_{2} + 3t_{2} + 4n_{3}+2.$$
Recall that we have the following combinatorial count:
\begin{equation}
\label{combcount}
4\cdot \binom{k}{2} = n_{2} + 2t_{2}+3n_{3}.
\end{equation}
Combining these two equations we get
$$d_{1}^{2}-d_{1}(2k-1)+(2k-1)^{2} = t_{2} + n_{3} + 2(k^{2}-k +1).$$
Simple manipulations lead us to
\begin{equation}
\label{disc}
d_{1}^{2}-d_{1}(2k-1) + 2k^{2}-2k - 1 - (t_{2}+n_{3})=0.
\end{equation}
If $\mathcal{C}$ is plus-one generated, then the discriminant $\triangle_{d_{1}}$ of \eqref{disc} satisfies
$$\triangle_{d_{1}} = (2k-1)^{2}-4\bigg(2k^{2}-2k-1 -(t_{2}+n_{3})\bigg) \geq 0.$$
This gives us
$$t_{2}+n_{3}\geq k^{2}-k-\frac{5}{4}.$$
Observe that
$$4\cdot\binom{k}{2} = 2(k^{2}-k) = n_{2} + n_{3} +2(t_{2}+n_{3}) \geq n_{2}+n_{3} + 2(k^{2}-k) - \frac{5}{2},$$
and we finally get
$$0 \leq n_{2} + n_{3} \leq 2.$$
On the other hand, 
$$4\cdot \binom{k}{2} = n_{2} + 2t_{2} + 3n_{3} \leq 2t_{2} + 3(n_{2}+n_{3}) \leq 2t_{2} + 6,$$
so we have
\begin{equation}
\label{t2}
t_{2}\geq k(k-1)-3.
\end{equation}
Using these combinatorial constraints we see that for $k=2$ one has $t_{2}\geq 0$ and $n_{2}+n_{3}\leq 2$, and we will return to this case in a moment. Now suppose that $k\geq 3$. Using the same argument as in \cite[Remark 3.3]{Pokora} (including the data on the triple intersections), we can observe that
\begin{equation}
\label{tacM}
t_{2} \leq \frac{4}{9}k^{2} + \frac{4}{3}k,
\end{equation}
and we arrive at the following chain of inequalities:
$$k(k-1)-3\leq t_{2} \leq \frac{4}{9}k^{2} + \frac{4}{3}k.$$
This gives us that $k \leq 5$.
Let us consider the case $k=5$. Using our bound on the number of tacnodes, we obtain
$$t_{2} \geq k(k-1) -3 = 17.$$
By \eqref{tacM} we see that for $k=5$ we have $t_{2} \leq 17$, so from now on we will assume that $\mathcal{C}$ is plus-one generated such that $t_{2}=17$. Note that the following weak combinatorics can only occur:
$$(n_{2},t_{2},n_{3}) \in \{(6,17,0), (3,17,1), (0,17,2)\},$$
so in order to get a plus-one generated example one needs to decide whether there exists an arrangement of $k=5$ conics with $t_{2}=17$ and $n_{3}=2$. By \cite[Theorem B]{Pokora}, the following Hirzebruch-type inequality holds (when $k \geq 3)$:
\begin{equation}
\label{hirzpok}
8k + n_{2} + \frac{3}{4}n_{3}\geq \frac{5}{2}t_{2}.
\end{equation}
However, if we plug $(k; n_{2},t_{2},n_{3}) = (5;0,17,2)$ into \eqref{hirzpok}, then we get a contradiction, which means that such an arrangement cannot exist.

To complete our proof, we need to show that for each $k \in \{2,3,4\}$ we have an example of a plus-one generated arrangement.
\begin{itemize}
\item[$(k=2)$] In this case we should have $n_{2}+n_{3}\leq 2$ and $t_{2}\geq 0$. If we assume that $t_{2}=2$, then the arrangement is nearly-free by \cite{DJP}, so we can exclude this case. The next possible case is $t_{2}=1$ and $n_{2}=2$, and in this situation we have $\tau(\mathcal{C})=5$. We show that such a weak combinatorics leads to a plus-one generated example. Let us take
$$\mathcal{Q}(x,y,z) =(x^2 + y^2 - z^2 )\cdot\bigg(x^2 - \frac{13}{10}xz + \frac{36}{10}y^2 - \frac{23}{10}z^2 \bigg). $$ 
as the defining equation of our arrangement. We can check using \verb}SINGULAR} that $(d_{1},d_{2},d_{3}) = (2,2,3)$, and this implies that $\mathcal{C}$ is plus-one generated.
\item[$(k=3)$] Using \cite[Proposition 5]{GM}, if we have an arrangement of $3$ conics with $5$ tacnodes, then these three conics are projectively equivalent to the three conics given by the equations:
$$\quad \quad \quad x^{2}+y^{2}-z^{2}=0, \, \ell^{2}x^{2}+(\ell^{2}+1)y^{2}-2\ell yz =0, \, m^{2}x^{2}+(m^{2}+1)y^{2}-2m yz = 0,$$
where $\ell, m \in \mathbb{C}\setminus \{0, \pm 1\}$, $\ell\neq m$, $\ell m\neq 1$. Let us take $\ell = 2$, $m=-2$, and denote by $Q(x,y,z)$ the defining equations of these three conics. Then $\tau(\mathcal{C})= 17$, since $t_{2}=5$, $n_{2}=2$, and $(d_{1},d_{2},d_{3}) = (3,3,4)$, so $\mathcal{C}$ is plus-one generated.
\item[$(k=4)$] Using \cite[Proposition 7 b)]{GM}, if we have an arrangement of $4$ conics with $11$ tacnodes, then these four conics are projectively equivalent to the conics given by the following equations:
$$x^2 + y^2 - z^2 =0, \, x^2 + r^{2}y^2 - r^{2}z^2=0, \, x^2 + (r^{2}+1)y^{2} \pm 2ryz,$$
where $r \in \mathbb{C} \setminus \{0,\pm 1, \pm \iota\}$. Let us take $r=2$ and denote by $Q(x,y,z)$ the defining equation of our arrangement $\mathcal{C}$. We have $\tau(\mathcal{C})=35$, since $t_{2}=11$ and $n_{2}=2$. Using \verb}Singular} we can check that $(d_{1},d_{2},d_{3}) = (4,4,5)$, which tells us that $\mathcal{C}$ is plus-one generated.
\end{itemize}
This completes the proof.
\end{proof}

 If we allow to have arrangements of conics and lines, then we can find more examples of plus-one generated arrangements. Let us start with the case when we have only double intersection points.
\begin{proposition}
\label{double}
Let $\mathcal{CL} \subset \mathbb{P}^{2}_{\mathbb{C}}$ be an arrangement consisting of $k\geq 1$ conics and $d\geq 1$ lines. Assume that $\mathcal{CL}$ has only $n_{2}$ double intersection points and is plus-one generated with $d_{3} > d_{2}$. Then $(k,d;n_{2})=(1,2;5)$. In other words, we have exactly one weak combinatorics for conic-line arrangements with only double intersection points that leads to a plus-one generated arrangement.
\end{proposition}
\begin{proof}
Denote by $m=2k+d$ the degree of the arrangement. Since $\mathcal{CL}$ is plus-one generated with the exponents $(d_{1},d_{2},d_{3})$ such that $d_{1}\leq d_{2} <  d_{3}$, and $d_{1}+d_{2} = m$, then by \cite[Theorem 2.1]{DimcaSernesi} one has
$$\frac{m}{2} \geq d_{1} \geq m-2,$$
and this follows from the fact that the log-canonical threshold for nodes is equal to $1$. This implies that $3 \leq m\leq 4$. If $m=3$, then we have $k=1$ and $d=1$, and an easy inspection shows us that the only possible case is to have $n_{2}=2$. However, such an arrangement is nearly-free, i.e., $d_{3}=d_{2}$, so we exclude this case. Let us pass to the case when $m=4$. It means that we have $k=1$ and $d=2$. Using the B\'ezout theorem, we must have exactly five nodes. Now we are going to give a geometric realization of the weak combinatorics $(k,d;n_{2}) = (1,2;5)$. Let us consider arrangement $\mathcal{C}$ defined by the following polynomial:
$$Q(x,y,z) = xy(x^{2} + y^{2} - z^{2}).$$
We have exactly $n_{2}=5$ nodes, and using \verb}SINGULAR} we can check that
$$(d_{1},d_{2},d_{3})=(2,2,3),$$ so $\mathcal{C}$ is plus-one generated.
\end{proof}
Now we pass to arrangements of $k\geq 1$ conics and $d\geq 1$ lines such that these admit $n_{2}$ nodes, $t_{2}$ tacnodes, and $n_{3}$ ordinary triple points. We will need the following general result.
\begin{proposition}
\label{degreeb}
Let $C \, : f = 0$ be a plus-one generated curve of degree $m$ admitting only nodes, tacnodes, and ordinary triple points as singularities. Then
$$\frac{m}{2}\geq d_{1} \geq \frac{2}{3}m-2.$$
In particular, we have that $m\leq 12$.
\end{proposition}
\begin{proof}
It follows from \cite[Proposition 4.7]{DimcaPokora}.
\end{proof}
Using this result, we can provide a degree-wise classification of a certain class of plus-one generated conic-line arrangements.
\begin{theorem}
\label{thm:CL_arr}
Let $\mathcal{CL} \subset \mathbb{P}^{2}_{\mathbb{C}}$ be an arrangement of $k\geq 1$ conics and $d\geq 1$ lines admitting only $n_{2}$ nodes, $t_{2}$ tacnodes, and $n_{3}$ ordinary triple points. Assume furthermore that $\mathcal{CL}$ is plus-one generated with $\nu(\mathcal{CL})=2$. Then $m:=2k+d \in \{4,5,6,7,8,9,10\}$, possibly except the cases $m=9$ or $m=10$.
\end{theorem}
\begin{proof}
Using Proposition \ref{degreeb} and Proposition \ref{double} we see that $m \in \{4, ..., 12\}$. 
We start with a degree-wise characterization. For $m=4$ we have a plus-one generated arrangement described in Proposition \ref{double}, so we are going to present constructions of plus-one generated conic-line arrangements in degrees $m\in \{5,6,7,8\}$ with types of singularities prescribed above.
\begin{itemize}
\item[$(m=5)$] Let $\mathcal{CL}_{5}$ be defined by the following polynomial:
$$Q(x,y,z) = x(x+y)(x-y)(x^2 + y^2 - z^2).$$ Here we have $n_{3}=1$ and $n_{2}=6$, so $\tau(\mathcal{CL}_{5})=10$. Then we can check directly, using \verb}SINGULAR}, that $(d_{1},d_{2},d_{3}) = (2,3,4)$, so $\mathcal{CL}_5$ is plus-one generated.
\item[$(m=6)$]  Let $\mathcal{CL}_{6}$ be defined by the following polynomial:
$$Q(x,y,z) = x(y-x)(y+x)(x^2 + y^2 - z^2)\bigg(y-\frac{\sqrt{2}}{2}z\bigg).$$ 
Here we have $n_{3}=3$ and $n_{2}=5$, so $\tau(\mathcal{CL}_{6})=17$. Then we can check directly, using \verb}SINGULAR}, that $(d_{1},d_{2},d_{3}) = (3,3,4)$, so $\mathcal{CL}_6$ is plus-one generated.
\item[$(m=7)$]  Let $\mathcal{CL}_{7}$ be defined by the following polynomial:
$$Q(x,y,z) = x(y-x)(y+x)(x^2 + y^2 - z^2)\bigg(y-\frac{\sqrt{2}}{2}z\bigg)\bigg(y+\frac{\sqrt{2}}{2}z\bigg).$$ 
Here we have $n_{3}=5$ and $n_{2}=5$, so $\tau(\mathcal{CL}_{6})=25$. Then we can check directly, using \verb}SINGULAR}, that $(d_{1},d_{2},d_{3}) = (3,4,5)$, so $\mathcal{CL}_7$ is plus-one generated.
\item[$(m=8)$] Let us consider the arrangement $\mathcal{CL}_{8}$ that is given by the following defining polynomial:
$$Q(x,y,z) = (x-y)(x+y)(x-z)(x+z)(y-z)(y+z)(x^{2} + y^{2} - z^{2}).$$ 
It is easy to see that we have $n_{2}=7$, $t_{2}=4$ and $n_{3}=4$, which gives us $\tau(\mathcal{CL}_{8}) = 35$. Using \verb}SINGULAR} we can check that $(d_{1},d_{2},d_{3}) = (4,4,5)$, so $\mathcal{CL}_{8}$ is plus-one generated.
\end{itemize}
Now we are going to exclude the existence of arrangements with $m \in \{11,12\}.$
\begin{itemize}
\item[$(m=11)$] Using Proposition \ref{degreeb}, we see that
$$\frac{11}{2} \geq d_{1} \geq \frac{22}{3}-2,$$
and since $d_{1} \in \mathbb{Z}_{\geq 0}$, we arrive at a contradiction.
\item[$(m=12)$] We are going to use two important combinatorial constraints. First of all, we have the naive combinatorial count:
\begin{equation}
\label{clcount}
\binom{12}{2} - k = \binom{m}{2}-k = n_{2}+2t_{2}+3n_{3}.
\end{equation}
Next, we will use the following Hirzebruch-type inequality \cite[Proposition 4.4]{Gal}:
\begin{equation}
\label{hirpp}
8k +n_{2} + n_{3} \geq 8k+n_{2} + \frac{3}{4}n_{3} \geq d+ \frac{5}{2}t_{2}.
\end{equation}
By the assumption our arrangements are plus-one generated, so using Proposition \ref{degreeb} we see that
$$6 = \frac{m}{2} \geq d_{1} \geq \frac{2}{3}\cdot 12 - 2 = 6,$$
so we arrive at the case $d_{1}=d_{2}=6$ and $d_{3}>6$. By Proposition \ref{dimspl} we get
\begin{equation}
\label{dpw}
89 = d_{1}^{2}-d_{1}(m-1)+(m-1)^2 - \nu(\mathcal{CL}) = n_{2}+3t_{2}+4n_{3}.
\end{equation}
Combining \eqref{dpw} with the (naive) combinatorial count, we arrive at
\begin{equation}
t_{2}+n_{3}=23+k, \quad n_{2}+n_{3}=20-3k.
\end{equation}
From Hirzebruch-type inequality we get
\begin{equation}
5k+20 = 8k + n_{2} + n_{3} \geq d + \frac{5}{2}t_{2},
\end{equation}
so we have found the following upper-bound on the number of tacnodes:
$$t_{2} \leq \frac{2}{5}\bigg(5k+20\bigg) - \frac{2}{5}d = 2k+8-\frac{2}{5}d.$$
Using the above constraints, we have the following possibilities:
\begin{center}
\begin{tabular}{ |c|c|c|c|c| } 
$k$ & $d$ & $n_{3}\leq$ & $t_{2}\leq $ &$n_{3}+t_{2}\leq $ \\
 \hline \hline
 1 & 10 & 17 & 6  & 23 \\ 
 2 &  8 & 14 & 8  & 22 \\ 
 3 &  6 & 11 & 11 & 22 \\
 4 &  4 & 8  & 14 & 22 \\ 
 5 &  2 & 5  & 17 & 22 \\
 \hline
\end{tabular}\,\,.
\end{center}
Since $t_{2}+n_{3} = 23 + k \geq 24$, we arrive at a contradiction.
\end{itemize}
\end{proof}

Here we would like give an interpretation of the sequence of inclusions 
$$\mathcal{CL}_5 \subset \mathcal{CL}_6 \subset \mathcal{CL}_7$$
of the examples constructed in Theorem \ref{thm:CL_arr}. This phenomenon can be explained by the next addition type result for plus-one generated conic-line arrangements.
Let us formulate our result in a broader setting, for curves $\mathcal{C}$ that have quasi-homogeneous singularities (i.e. for all singular points $p \in Sing(\mathcal{C})$ we have $\tau_{p}=\mu_{p}$), as it may be applied in further research.

\begin{prop}
\label{prop:POG_U_line_cond}
Let $\mathcal{C}'$ be a plus-one generated curve with defect $\nu(\mathcal{C}')=2$ and the exponents $(d_{1}',d_{2}')$, and let $L$ be a line that is not an irreducible component of $\mathcal{C}'$ such that the curve $\mathcal{C}:= \mathcal{C}' \cup L$ is again plus-one generated with defect $\nu(\mathcal{C})=2$. Suppose that both $\mathcal{C}'$ and $\mathcal{C}$ have quasi-homogeneous singularities. Then $|\mathcal{C}' \cap L| \in \{d'_1+1, \; d'_2\}$.
\end{prop}
\begin{proof}
Let us  denote by $\alpha_L$ a homogeneous linear form that defines the line $L$. By \cite[Theorem 1.6, Remark 1.8]{STY}, we have the following exact sequence of vector bundles
\begin{equation}
\label{eq:triple_sheaves}
0 \rightarrow \mathcal{E}_{\mathcal{C}'}(-1) \overset{\alpha_L}\longrightarrow  \mathcal{E}_{\mathcal{C}} \rightarrow \mathcal{O}_L(1-|\mathcal{C}' \cap L|) \rightarrow 0
\end{equation}
where $\cE_{\mathcal{C}}$ denotes the locally free sheaf (i.e. vector bundle) obtained by the sheafification of the graded module $D_0(\mathcal{C})$. We will compute $|\mathcal{C}' \cap L|$ via Chern numbers associated to the vector bundles in sequence \eqref{eq:triple_sheaves}. For a complex vector bundle $\mathcal{E}$, we denote by $c_i(\mathcal{E})$ is the $i$-th Chern number of $\mathcal{E}$.
The exact sequence \eqref{eq:triple_sheaves} implies the equality
\begin{equation}
\label{eq:chern2}
c_2(\mathcal{E}_{\mathcal{C}}) = c_2( \mathcal{E}_{\mathcal{C}'}(-1) ) + c_1( \mathcal{E}_{\mathcal{C}'}(-1) )c_1(\mathcal{O}_L(1-|\mathcal{C}' \cap L|)) + c_2(\mathcal{O}_L(1-|\mathcal{C}' \cap L|)),
\end{equation}
see for instance \cite[Section 3]{F} for details. 

If $\mathcal{C}$ is plus-one generated with exponents $(d_1, d_2)$ and level $d_3$, 
then we get, directly from Definition \ref{def:POG_curves}, formulas for the Chern numbers of the associated vector bundle $\mathcal{E}_{\mathcal{C}}$, in terms of the exponents and level. We will leave the details of this easy computation to the interested reader (check \cite{F} for the general theory and computations on Chern numbers). We have:
\begin{equation}
\label{eq:chern1_POG}
 c_1( \mathcal{E}_{\mathcal{C}}) = 1 - d_1 - d_2.
 \end{equation}
and 
\begin{equation}
\label{eq:chern2_POG}
c_2(\cE_{\mathcal{C}}) = d_1(d_2-1)+d_3-d_2+1.
\end{equation}

Since $\nu(\mathcal{C}') = \nu(\mathcal{C}) = 2$ in our hypothesis, we have $d'_3-d'_2 = d_3-d_2 = 1$.
Now we will substitute the formulas for Chern numbers \eqref{eq:chern1_POG} and \eqref{eq:chern2_POG} in the equality \eqref{eq:chern2}, taking into account also the formula for the Chern polynomial of $\mathcal{E}_{\mathcal{C}'}(-1) = \mathcal{E}_{\mathcal{C}'} \otimes _{\mathcal{O}_{\mathbb{P}^2}} \mathcal{O}_{\mathbb{P}^2}(-1)$ in terms of the Chern polynomial of $\mathcal{E}_{\mathcal{C}'}$, see again \cite[Section 3]{F}. We obtain the equality

$$c_2(\cE_{\mathcal{C}}) = (c_2(\cE_{\mathcal{C}'}) - c_1(\cE_{\mathcal{C}'})+ 1) + (-2+c_1(\cE_{\mathcal{C}'})) + c_2(\mathcal{O}_L(1-|\mathcal{C}' \cap L|)),$$
or explicitly
\begin{equation}
\label{eq:exponents}
d_1(d_2-1)+2 = \bigg(d'_1(d'_2-1)+2 - (1 - d'_1-d'_2) + 1 \bigg) + \bigg(-2+(1 - d'_1-d'_2)\bigg)+|\mathcal{C}' \cap L| .
\end{equation}
Observe that by \cite[Proposition 3.1]{DIS} we have $ d_1 \in \{d'_1, d'_1+1\}$. To complete our proof, let us first consider the case $d_1 = d'_1$. Then $d_2 = d'_2 +1$ and \eqref{eq:exponents} becomes 
$$|\mathcal{C}' \cap L| = d'_1+1.$$
Secondly, if  $d_1 = d'_1+1$, then $d_2 = d'_2 $, and  \eqref{eq:exponents} becomes 
$$|\mathcal{C}' \cap L| = d'_2.$$
In conclusion, we necessarily get $|\mathcal{C}' \cap L| \in \{d'_1+1, \; d'_2\}$, which completes the proof.
\end{proof}

\begin{remark}
\label{rem:POG_add}
The above Proposition \ref{prop:POG_U_line_cond} states, in particular, that one can obtain from a plus-one generated conic-line arrangement $\mathcal{C}'$ with defect $2$, by adding a line $L$, a new plus-one generated conic-line 
arrangement $\mathcal{C}' \cup \{L\}$ with defect $2$ {\bf only if} the line $L$ intersects $\mathcal{C}'$ in a certain number of points, depending on the exponents of $\mathcal{C}'$.
We use this geometric obstruction when we search for lines to obtain $\mathcal{CL}_6$ from $\mathcal{CL}_5$, respectively 
$\mathcal{CL}_7$ from $\mathcal{CL}_6$, in the proof of Theorem \ref{thm:CL_arr}.

Proposition \ref{prop:POG_U_line_cond} also elucidates why the conic-line arrangement $\mathcal{CL}_7$ from Theorem \ref{thm:CL_arr} cannot be extended by adding a line to a plus-one generated arrangement with the same restrictions on the types of multiple points and defect $2$. In this case it is easy to see that one cannot add to the conic-line arrangement $\mathcal{CL}_7$  a new line $L$ with the property that such a line intersects $\mathcal{CL}_7$ in $4$ points, while preserving the restrictions on the type of singularities for a new conic-line arrangement $\mathcal{CL}_7 \cup \{L\}$. 


Likewise, it is easy to see that one cannot add to the conic-line arrangement $\mathcal{CL}_8$ a new line $L$ such that $L$ intersects $\mathcal{CL}_8$ in $4$ or $5$ points while preserving the restrictions on the type of singularities for a new plus-one generated conic-line arrangement $\mathcal{CL}_8 \cup \{L\}$ with defect 2.
\end{remark}

\begin{remark}
\label{rem:not_sufficient_cond}
The necessary geometric condition on the line $L$ from Proposition \ref{prop:POG_U_line_cond} is not a {\bf sufficient} condition, as we can see in the next example. Take $L: 2y - \epsilon z = 0$ with $\epsilon^2 = 2$. Consider the conic-line arrangement $\mathcal{CL}:= \mathcal{CL}_8 \cup \{L\}$. Recall from the proof of Theorem \ref{thm:CL_arr} that $\mathcal{CL}_8$ is plus-one generated with exponents $(4,4)$ and defect $2$. Although $|\mathcal{CL}_8 \cap L| = 5$, the arrangement $\mathcal{CL}$ is not plus-one generated, as a \verb}SINGULAR} computation shows.
\end{remark}

Concluding our investigations, let us point out here that in the case of $m \in \{9,10\}$, our methods are not sufficient to decide on the existence/non-existence of such plus-one generated conic-line arrangements since, as usually, such boundary cases are very difficult to handle. For instance, if we assume that $k=1$ and $d=7$, then one has to decide whether the following weak combinatorics are realizable over the complex numbers:
$$(k,d; n_{2},t_{2},n_{3}) \in \bigg\{(1,7;3,1,10),(1,7;4,2,9),(1,7;5,3,8),(1,7;6,4,7),(1,7;7,5,6)\bigg\}.$$
Here we show how to exclude the existence of the weak combinatorics $(k,d;n_{2},t_{2},n_{3})=(1,7;7,5,6)$. In order to do so, we need the following general result.
\begin{theorem}
\label{ntt}
Let $C \subset \mathbb{P}^{2}_{\mathbb{C}}$ be a reduced curve of degree $m\geq 9$ admitting $n_{2}$ nodes, $t_{2}$ tacnodes, and $n_{3}$ ordinary triple points, then for a real number $\alpha \in [1/3, 2/3]$ one has
\begin{equation}
\label{ientt}
    (6\alpha -3\alpha^2)\cdot n_{2} + \bigg(-6\alpha^2 + 15\alpha - \frac{3}{8}\bigg)\cdot t_{2} + \bigg(-\frac{27}{4}\alpha^{2}+18\alpha\bigg)\cdot n_{3} \leq (3\alpha - \alpha^2)m^2 - 3\alpha m.
\end{equation}
\end{theorem}
\begin{proof}
We are going to use directly an orbifold version of the Bogomolov-Miyaoka inequality. We will work with the pair $\bigg(\mathbb{P}^{2}_{\mathbb{C}}, \alpha D \bigg)$ which is going to be log-canonical and effective. In order to be effective, one requires that $\alpha \geq \frac{3}{{\rm deg}(C)} = \frac{3}{m}$, and in order to be log-canonical, $\alpha$ should be less than or equal to the minimum of log-canonical thresholds of our singular points, which means that $\alpha \leq {\rm min}\bigg\{1,\frac{3}{4},\frac{2}{3}\bigg\}$. Summing up, based on the first part of our discussion, let $\alpha \in [3/m, 2/3]$. Now we are in a position that allows us to use the following Langer's inequality proved in \cite{Langer}, namely 
\begin{equation}
\label{logMY}
\sum_{p \in {\rm Sing}(\mathcal{C})}  3\bigg( \alpha \bigg(\mu_{p} - 1\bigg) + 1 - e_{orb}\bigg(p,\mathbb{P}^{2}_{\mathbb{C}}, \alpha D\bigg) \bigg) \leq (3\alpha -\alpha^2)m^{2} - 3\alpha m,
\end{equation}
where $\mu_{p}$ is the local Milnor number of $p \in {\rm Sing}(C)$ and $e_{{\rm orb}}\bigg(p,\mathbb{P}^{2}_{\mathbb{C}}, \alpha  D\bigg)$ denotes the local orbifold Euler number of $p \in {\rm Sing}(C)$. In our setting we have the following values:
\begin{table}[ht]
\centering
\begin{tabular}{|c|c|c|}
type   &  $e_{orb}(p,\mathbb{P}^{2}_{\mathbb{C}},\alpha D)$ & $\alpha$ \\ \hline \hline
$A_{1}$ &  $(1-\alpha)^{2}$ & $0 < \alpha \leq 1$ \\
$A_{3}$ &  $\frac{(3-4\alpha)^{2}}{8}$ & $1/4 \leq \alpha \leq 3/4$  \\
$D_{4}$ & $\frac{(2-3\alpha)^{2}}{4}$ & $0<\alpha \leq 2/3$\\ \hline
\end{tabular}\,\,.
\end{table}
\newpage
Assume that $\alpha \in [1/3, 2/3]$, then our inequality follows from plugging the collected data above into \eqref{logMY}.
\end{proof}
\begin{corollary}
There does not exists a conic-line arrangement $\mathcal{CL}$ in $\mathbb{P}^{2}_{\mathbb{C}}$ having the weak combinatorics $(k,d;n_{2},t_{2},n_{3}) =(1,7;7,5,6)$.
\end{corollary}
\begin{proof}
It follows from Theorem \ref{ntt}, namely we can take $\alpha = \frac{4}{10}$ and then we can check that \eqref{ientt} does not hold for $(k,d;n_{2},t_{2},n_{3}) =(1,7;7,5,6)$.
\end{proof}
\section*{Acknowledgments}
We would like to thank Xavier Roulleau for sharing his example of a conic-line arrangement that was used in Example \ref{XR} and for his help with symbolic computations in \verb}MAGMA}. We would like to thank an anonymous referee for an interesting question that motivated us to prove Proposition \ref{prop:POG_U_line_cond}.

Piotr Pokora was partially supported by The Excellent Small Working Groups Programme \textbf{DNWZ.711/IDUB/ESWG/2023/01/00002} at the Pedagogical University of Cracow. 

Anca M\u{a}cinic was partially supported by a grant of the Romanian Ministry of Education and Research, CNCS - UEFISCDI, project number \textbf{PN-III-P4-ID-PCE-2020-2798}, within PNCDI III.

\bigskip

Anca~M\u acinic,
Simion Stoilow Institute of Mathematics, 
P.O. Box 1-764, RO-014700 Bucharest, Romania. \\
\nopagebreak
\textit{E-mail address:} \texttt{amacinic@imar.ro}

\bigskip
Piotr Pokora,
Department of Mathematics,
University of the National Education Commission Krakow,
Podchor\c a\.zych 2,
PL-30-084 Krak\'ow, Poland. \\
\nopagebreak
\textit{E-mail address:} \texttt{piotrpkr@gmail.com, piotr.pokora@up.krakow.pl}
\end{document}